\documentclass[12pt]{amsart}

\usepackage{amsmath,amssymb,amsthm}
\usepackage{cite}
\usepackage[margin=1in]{geometry}
\usepackage{subcaption}
\usepackage[dvipsnames]{xcolor}
\usepackage{longtable}
\usepackage{verbatim} 
\usepackage{ytableau}
\ytableausetup{boxsize=2em}
\usepackage{booktabs}
\usepackage{tikz}
\usetikzlibrary{positioning}
\usepackage{cleveref}

\DeclareMathOperator{\inn}{in}
\DeclareMathOperator{\out}{out}
\DeclareMathOperator{\col}{col}

\DeclareMathOperator{\leftof}{left}
\DeclareMathOperator{\row}{row}

\newcommand{\met}[1]{\langle #1\rangle}

\DeclareSymbolFont{sfoperators}{OT1}{cmss}{m}{n}
\DeclareSymbolFontAlphabet{\mathsf}{sfoperators}

\makeatletter
\def\operator@font{\mathgroup\symsfoperators}
\makeatother

\renewcommand{\c}[2]{ 
    \ifcase#1
    \or \textcolor{BrickRed}{#2}
    \or \textcolor{ForestGreen}{#2}
    \or \textcolor{RoyalBlue}{#2}
    \or \textcolor{Thistle}{#2}
    \or \textcolor{black}{#2}
    \or \textcolor{gray}{#2}
    \or \textcolor{Tan}{#2}
    \fi
}
\newcommand{\stack}[5]{
    \begin{tikzpicture}[scale=0.6]
        \draw[thick] (-0.5,0) -- (1,0) -- (1,-2) -- (2,-2) -- (2,0) -- (3.5,0);
        \node[fill = white, draw = white] at (2.8,.5) {$#5$};
        \node[fill = white, draw = white] at (1.5,-1.5) {$#4$};
        \node[fill = white, draw = white] at (1.5,-.8) {$#3$};
        \node[fill = white, draw = white] at (1.5,-.1) {$#2$};
        \node[fill = white, draw = white] at (.2,.5) {$#1$};
    \end{tikzpicture}
}
\newcommand{\arrow}{
    \!\!\!\!\!\!\!
    \begin{tikzpicture}[scale=0.6]
        \node at (0,0) {};
        \node at (0,0.7) {$\longrightarrow$};
    \end{tikzpicture}
    \!\!\!\!\!\!\!
}

\theoremstyle{definition}
\newtheorem{theorem}{Theorem}[section]
\newtheorem{example}[theorem]{Example}
\newtheorem{lemma}[theorem]{Lemma}
\newtheorem{corollary}[theorem]{Corollary}
\newtheorem{definition}[theorem]{Definition}
\newtheorem{proposition}[theorem]{Proposition}
\newtheorem{conjecture}[theorem]{Conjecture}
\newtheorem{remark}[theorem]{Remark}

\title{Asymptotics of the Average Stack-Sorting Depth}
\author{Jerry Zhang}
\address{\textsc{J. Zhang}, South Pasadena High School, South Pasadena, CA, 91030}
\email{jerrylezhang@gmail.com}

\begin{document}

\begin{abstract}
Let $\mathcal{D}_n$ denote the average number of passes of the stack-sorting map $s$ required to sort a permutation in $S_n$. We use the recently introduced framework of stack-sorting diagrams and tableaux to prove that the limit $\lim_{n\to\infty}\mathcal{D}_n/n$ exists. This resolves a longstanding conjecture of West originally proposed in $1990$. As a consequence, we also provide a monotonically increasing sequence that converges to $\lim_{n\to\infty}\mathcal{D}_n/n$, improving upon Defant's lower bound of $\lambda\approx 0.62433$.
\end{abstract}

\maketitle

\section{Introduction}
West \cite{west} introduced the stack-sorting map $s$ in 1990 as a deterministic procedure on permutations. The input permutation $\pi$ is read from left to right, and at each step, the next entry of $\pi$ is pushed into the stack if either the stack is empty or the top element of the stack is greater than the next entry. If the next entry of $\pi$ cannot be pushed into the stack, then the top element of the stack is popped out and appended to the output permutation. The procedure concludes when all entries have been appended to the output permutation. Figure \ref{fig:westmap} illustrates the map on $\pi=1324$. Note that each application of $s$ sends the greatest unsorted entry to its correct position. Thus, at most $n-1$ applications of $s$ are required to sort a permutation with $n$ elements. The \emph{stack-sorting depth} $\met{\pi}$ is the minimum number of passes required to sort $\pi$.
\begin{figure}[h]
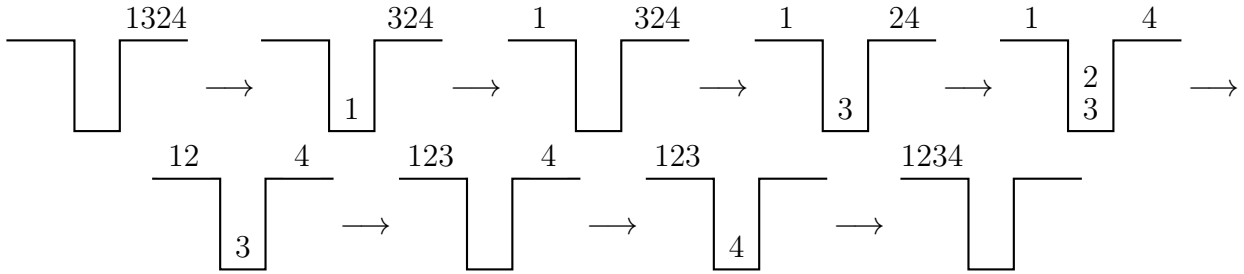

    \begin{center}
        \stack{}{}{}{}{1324}\arrow
        \stack{}{}{}{1}{324}\arrow
        \stack{1}{}{}{}{324}\arrow
        \stack{1}{}{}{3}{24}\arrow
        \stack{1}{}{2}{3}{4}\arrow
        \stack{12}{}{}{3}{4}\arrow
        \stack{123}{}{}{}{4}\arrow
        \stack{123}{}{}{4}{}\arrow
        \stack{1234}{}{}{}{}
    \end{center}
    \caption{West's stack-sorting map applied on $\pi = 1324$.}
    \label{fig:westmap}
\end{figure}

Define the average stack-sorting depth over $S_n$ to be
\[\mathcal{D}_n:=\frac{1}{n!}\sum_{\pi\in S_n} \met{\pi}.\]
West \cite{west} provided the first asymptotic lower bound for $\mathcal{D}_n$ with a pattern avoidance argument, and he conjectured that $\mathcal{D}_n/n$ converges as $n$ approaches infinity.
\begin{theorem}[West]
\[\liminf_{n\to\infty}\frac{\mathcal{D}_n}{n}\ge 0.23.\]
\end{theorem}
In 2020, Defant~\cite{defant2} improved upon West's lower bound and proved the first nontrivial upper bound.
\begin{theorem}[Defant]
\[\lambda\le \liminf_{n\to\infty}\frac{\mathcal{D}_n}{n}\le \limsup_{n\to\infty}\frac{\mathcal{D}_n}{n}\le \frac{3}{5}(7-8\ln 2)\]
where $\lambda$ is the Golomb--Dickman constant.
\end{theorem}
Our main result is a proof of the convergence of $\mathcal{D}_n/n$ as $n$ approaches infinity.  

\begin{theorem}\label{thm:main}
The limit $\lim_{n\to\infty}\mathcal{D}_n/n$ exists.
\end{theorem}

Let $S_n'$ be the set $\{\rho0:\rho\in S_n\}$, where $\rho0$ denotes the concatenation of $\rho$ and $0$, and define the quantity
\[\mathcal{D}'_n := \frac{1}{n!}\sum_{\rho\in S_n}\met{\rho0}.\]
Defant\cite{defant2} showed that the difference between $\mathcal{D}'_n/n$ and $\mathcal{D}_n/n$ goes to $0$ as $n$ approaches infinity. 
\begin{proposition}[{Defant \cite[Proposition~2.9]{defant2}}]\label{prop:defant}
We have
\[\lim_{n\to \infty}\left(\frac{\mathcal{D}'_n}{n}-\frac{\mathcal{D}_n}{n}\right) = 0.\]
\end{proposition}
Thus, a proof of the convergence of $\lim_{n\to\infty}\mathcal{D}'_n/n$ is sufficient for \Cref{thm:main}. As a natural consequence of the proof of \Cref{thm:main}, we also obtain a means to compute lower bounds arbitrarily close to $\lim_{n\to\infty}\mathcal{D}_{n}/n$.
\begin{corollary}\label{cor:lb}
$(\mathcal{D}_{n-1}'/n)_{n=2}^{\infty}$ is a monotonically increasing sequence, which converges to $\lim_{n\to\infty}\mathcal{D}_n/n$.
\end{corollary}

The current author has computed $\mathcal{D}'_{n-1}/n$ up to $n=38$, which yields the bound
\[0.724182\le \lim_{n\to\infty}\frac{\mathcal{D}_n}{n}\le \frac{3}{5}(7-8\ln 2).\]

\section{Preliminaries}\label{sec:prelim}

We follow the notation and conventions of \cite{zhang26} with some minor discrepancies.

\subsection*{Basic notation}

A \emph{composition} of $n$, written $\alpha\models n$, is a tuple $\alpha=(\alpha_1,\ldots,\alpha_k)$ of positive integers summing to $n$. The \emph{length} and \emph{width} of $\alpha$ are denoted by $\ell(\alpha):=k$ and $w(\alpha):=\max_{i\in[k]}\alpha_i$ respectively. The \emph{composition diagram} of $\alpha$ is the set \[D(\alpha):=\{(i,j)\in\mathbb{N}^2:i\le\alpha_j\}\] along with the partial order relation $\ge_{D}$ satisfying \[(i,j)\ge_D(i',j') \iff i\le i',\ j\le j'.\] A \emph{linear extension} of $(D(\alpha),\ge_D)$ is a bijection $f:[n]\to D(\alpha)$ such that $f(x)\ge_D f(y)$ implies $x\ge y$ for all $x,y\in [n]$. The set of all linear extensions of $D(\alpha)$ is denoted by $\mathcal{L}(D(\alpha))$.
\subsection*{Stack-sorting diagrams and tableaux}

For every permutation $\pi=\rho0\in S_n'$, we assign two coordinates $\col_{\pi}(i)$ and $\row_{\pi}(i)$ to each entry $i\in [n]$. The column $\col_{\pi}(i)$ is the unique integer for which $i$ is sent to the right of $0$ when $s$ is applied to $s^{\col_{\pi}(i)-1}(\pi)$. If $\col_{\pi}(i)=1$, then the row $\row_{\pi}(i)$ denotes the position of $i$ amongst the right-to-left maxima of $\rho$. Otherwise, $\row_{\pi}(i)$ is defined recursively by the equation
\[\row_{\pi}(i) = \row_{\pi}(\leftof_{\pi}(i)),\]
where $\leftof_{\pi}(i)$ is the first element of $s^{\col_{\pi}(i)-2}(\pi)$ to the right of $i$ such that \[\col_{\pi}(\leftof_{\pi}(i))=\col_{\pi}(i)-1.\]
The existence of $\col_{\pi}(i)$, $\row_{\pi}(i)$, and $\leftof_{\pi}(i)$ are proven in \cite{zhang26}. 
\begin{remark}
The general motivation for this framework is the observation that each element $i\in[n]$ is sent to the right of $0$ with iterative applications of $s$ on $\pi$, and the permutation is sorted precisely when every element has been sent to the right of $0$. The column $\col_{\pi}(i)$ tracks the iteration in which $i$ is sent to the right of $0$, while $\row_{\pi}(i)$ is responsible for tracking the relative position of $i$ with respect to elements of the current and previous column.
\end{remark}
We now introduce the stack-sorting diagram and the stack-sorting tableau of $\pi\in S_n'$. Let the composition $\alpha_\pi\models n$ be defined such that the $j$th part of $\alpha_{\pi}$ is
\[(\alpha_\pi)_j:=|\{i\in[n]:\row_\pi(i)=j\}|.\]
We refer to $D(\alpha_{\pi})$ as the \emph{stack-sorting diagram} of $\pi$, and the \emph{stack-sorting tableau} of $\pi$ is the map $T_{\pi}:[n]\to D(\alpha_{\pi})$ satisfying
\[T_{\pi}(i) := (\col_{\pi}(i), \row_{\pi}(i))\]
for all $i\in [n]$.
\begin{example}\label{ex:main}
Consider the permutation \[\pi = 9\ 3\ 10\ 7\ 8\ 2\ 6\ 1\ 4\ 5\ 0\in S_{10}'.\] We visualize $T_\pi$ in English notation below. Note that $w(\alpha_{\pi})=\met{\pi}=4$, and $\leftof_{\pi}(T_{\pi}^{-1}(i,j))$ corresponds to the entry $T_{\pi}^{-1}(i-1,j)$ for all $(i,j)\in D(\alpha_{\pi})$ such that $i>1$.
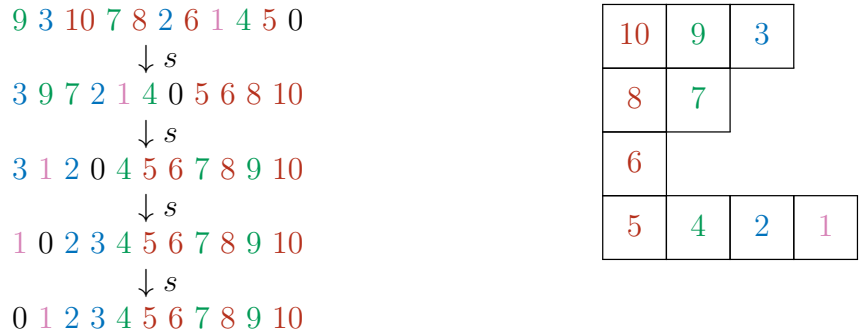
\begin{center}
\begin{minipage}[t]{0.45\textwidth}
\centering
\[
\begin{array}{c}
\c2{9}\ \c3{3}\ \c1{10}\ \c2{7}\ \c1{8}\ \c3{2}\ \c1{6}\ \c4{1}\ \c2{4}\ \c1{5}\ 0 \\
\downarrow{s} \\
\c3{3}\ \c2{9}\ \c2{7}\ \c3{2}\ \c4{1}\ \c2{4}\ 0\ \c1{5}\ \c1{6}\ \c1{8}\ \c1{10} \\
\downarrow{s} \\
\c3{3}\ \c4{1}\ \c3{2}\ 0\ \c2{4}\ \c1{5}\ \c1{6}\ \c2{7}\ \c1{8}\ \c2{9}\ \c1{10} \\
\downarrow{s} \\
\c4{1}\ 0\ \c3{2}\ \c3{3}\ \c2{4}\ \c1{5}\ \c1{6}\ \c2{7}\ \c1{8}\ \c2{9}\ \c1{10} \\
\downarrow{s} \\
0\ \c4{1}\ \c3{2}\ \c3{3}\ \c2{4}\ \c1{5}\ \c1{6}\ \c2{7}\ \c1{8}\ \c2{9}\ \c1{10}
\end{array}
\]
\end{minipage}
\begin{minipage}[t]{0.45\textwidth}
\centering
\[
\ytableaushort{{\c1{10}}{\c2{9}}{\c3{3}},{\c1{8}}{\c2{7}},{\c1{6}},{\c1{5}}{\c2{4}}{\c3{2}}{\c4{1}}}
\]
\end{minipage}
\end{center}
\end{example}
In general, the equality $w(\alpha_{\pi})=\met{\pi}$ holds because the number of columns in $D(\alpha_{\pi})$ equals the number of iterations of $s$ required to sort $\pi$. 
\begin{theorem}[{Zhang \cite[Theorem~3.16]{zhang26}}]\label{thm:Tpi_bijection}
For all $\pi\in S_n'$, the map $T_\pi$ is a linear extension of $(D(\alpha_\pi),\ge_D)$.
\end{theorem}
\subsection*{Hook length formula for stack-sorting}
For every composition $\alpha\models n$ and cell $(i,j)\in D(\alpha)$, let $u_{\alpha}(i,j)$ denote \[\max(\{j'\in [j-1]: \alpha_{j'}\ge i-1\}\cup \{0\}),\] and define the \emph{hook length} function $h_\alpha: \mathbb{N}^2\to \mathbb{N}$ such that
\[h_\alpha(i,j) := \bigl|\{(i',j')\in D(\alpha):i'\le i-1, u_{\alpha}(i,j)<j'\le j\}\bigr|\]
if $i>1$ and $h_{\alpha}(i,j):= 1$ if $i=1$ for all $(i,j)\in\mathbb{N}^2$. In words, if $i>1$, then $h_{\alpha}(i,j)$ is the number of cells between columns $1$ and $i-1$ and rows $u_{\alpha}(i,j) + 1$ and $j$.
\begin{example}
Continuing with \Cref{ex:main}, each cell $(i,j)\in D(\alpha)$ below is labeled with its hook length $h_\alpha(i,j)$ for the composition $\alpha=(3,2,1,4)$.
\[\ytableaushort{{1}{1}{2},{1}{1},{1},{1}{1}{3}{6}}\]
\end{example}
\begin{theorem}[Zhang {\cite[Corollary~4.11]{zhang26}}]\label{thm:hook_count}
For all $\alpha\models n$, the number of permutations $\pi\in S_n'$ satisfying $\alpha_{\pi}=\alpha$ is
\[N(\alpha):=|\mathcal{L}(D(\alpha))|\cdot\prod_{(i,j)\in D(\alpha)}h_\alpha(i,j).\]
\end{theorem}

\Cref{thm:hook_count} is arguably the most powerful property of stack-sorting diagrams. Two natural consequences are 
\[\frac{1}{n!}\sum_{\alpha\models n}N(\alpha) = 1\]
and
\[\frac{1}{n!}\sum_{\alpha\models n}N(\alpha)\,w(\alpha) = \mathcal{D}_{n}'.\]

\subsection*{Corners}
For every composition $\alpha\models n$, an \emph{inner corner} of $D(\alpha)$ is a minimal element of $(D(\alpha),\ge_D)$. Let $\inn(D(\alpha))$ denote the set of all inner corners of $D(\alpha)$, and for all $c=(\alpha_j, j)\in \inn(D(\alpha))$, define $\alpha^{-c}$ to be the composition obtained by decrementing $\alpha_j$ by one and removing part $j$ if it becomes zero. Likewise, for every composition $\beta\models n-1$, an \emph{outer corner} of $D(\beta)$ is a cell $c\in \mathbb{N}^2\setminus D(\beta)$ such that there exists a composition $\beta^{+c}\models n$ for which $c$ is an inner corner of $D(\beta^{+c})$ and $\beta = (\beta^{+c})^{-c}$. Let $\out(D(\beta))$ denote the set of all outer corners of $D(\beta)$.
\begin{example}
The inner and outer corners of $D(2,3,1)$ are labeled blue and red respectively in the diagram below.
\begin{figure}[ht]
\centering
\newcommand\dcell[3]{\draw[#3] (#1-1,-#2) rectangle (#1,1-#2);}
\begin{tikzpicture}[scale=0.85,
  base/.style  = {draw=black, line width=0.4pt},
  inner/.style = {draw=blue!55!black, line width=0.6pt, fill=blue!22},
  outer/.style = {draw=red!72!black, line width=0.8pt, dashed, fill=red!9}]
  \dcell{1}{1}{base}  \dcell{2}{1}{base}
  \dcell{1}{2}{base}  \dcell{2}{2}{base}
  \dcell{3}{2}{inner} \dcell{1}{3}{inner}
  \dcell{4}{2}{outer} \dcell{2}{3}{outer} \dcell{1}{4}{outer}
\end{tikzpicture}
\label{fig:corners}
\end{figure}
\end{example}

\section{Proof of the Main Result}\label{sec:main}
In this section, we prove \Cref{thm:main} and \Cref{cor:lb} by understanding $\mathcal{D}_n'$ in the language of stack-sorting diagrams. We begin by proving some recursive identities about $N(\alpha)$ and $h_{\alpha}$.

\begin{lemma}\label{lem:factorization}
For all $\alpha\models n$ and $c\in\inn(D(\alpha))$, it holds that
\[\prod_{(i,j)\in D(\alpha)}h_\alpha(i,j) \;=\; h_\alpha(c)\,\prod_{(i,j)\in D(\alpha^{-c})}h_{\alpha^{-c}}(i,j).\]
\end{lemma}
\begin{proof}
It suffices to prove $h_{\alpha}(i,j)=h_{\alpha^{-c}}(i,j)$ for all $(i,j)\in D(\alpha^{-c})$. Now, note that
\[\{j'\in [j-1]: \alpha_{j'}\ge i-1\} = \{j'\in [j-1]: \alpha^{-c}_{j'}\ge i-1\}\]
and
\[u_{\alpha}(i,j) = u_{\alpha^{-c}}(i,j)\]
because $c$ is a minimal element of $D(\alpha)$ and a difference between the two sets would imply $(i,j)\le_{D} c$. It subsequently follows that
\[c\notin \{(i',j')\in D(\alpha): i'\le i-1, u_{\alpha}(i,j)<j'\le j\},\]
so $h_{\alpha}(i,j)=h_{\alpha^{-c}}(i,j)$ as desired.
\end{proof}

\begin{lemma}\label{lem:corner}
For all $\alpha\models n$, it holds that
\[N(\alpha) = \sum_{c\,\in\,\inn(D(\alpha))} N(\alpha^{-c})\,h_\alpha(c).\]
\end{lemma}

\begin{proof}
By casework on the assignment of the least element, it is well known that the number of linear extensions of a poset $(D(\alpha),\ge_D)$ can be decomposed into 
\[|\mathcal{L}(D(\alpha))| = \sum_{c\in\inn(D(\alpha))}|\mathcal{L}(D(\alpha^{-c}))|.\]
Multiplying both sides by $\prod_{(i,j)\in D(\alpha)}h_\alpha(i,j)$ and applying \Cref{lem:factorization}, we obtain the identity
\begin{align*}
N(\alpha) &= \sum_{c\in\inn(D(\alpha))}\left(|\mathcal{L}(D(\alpha^{-c}))|\cdot \prod_{(i,j)\in D(\alpha^{-c})}h_{\alpha^{-c}}(i,j)\right)\cdot h_\alpha(c) \\
&= \sum_{c\in\inn(D(\alpha))} N(\alpha^{-c})\,h_{\alpha}(c).\qedhere
\end{align*}
\end{proof}

\begin{definition}\label{def:a}
For every composition $\beta\models n$, let $j_{\max}(\beta)$ denote the greatest index
$j$ for which $\beta_j=w(\beta)$, and define
\[t(\beta) := \sum_{j=1}^{j_{\max}(\beta)}\beta_j = \bigl|\{(i,j)\in D(\beta)\mid j\le j_{\max}(\beta)\}\bigr|.\]
Note that $t(\beta)$ counts the cells of $D(\beta)$ up to and including its last longest row.
\end{definition}

\begin{lemma}\label{lem:hooksum}
For all $\beta\models n-1$, it holds that
\[\sum_{c\,\in\,\out(D(\beta))} h_{\beta^{+c}}(c) = n\]
and
\[\sum_{\substack{c\,\in\,\out(D(\beta)) \\ w(\beta^{+c})=w(\beta)+1}} h_{\beta^{+c}}(c) = t(\beta).\]
\end{lemma}

\begin{proof}
Let $(j_t)_{t=1}^m$ denote the sequence of weak right-to-left maxima of $\beta$, and set $j_0:=0$. The outer corners of $D(\beta)$ are precisely $c_t:=(\beta_{j_t}+1,j_t)$ for $t\in[m]$ and $c_{m+1}:=(1,\ell(\beta)+1)$. Now, note that for all $t\in [m]$
\[u_{\beta^{+c_t}}(\beta_{j_t}+1,j_t) = \max\left(\{j'\in [j_t-1]:\beta_{j'}\ge \beta_{j_t}\}\cup\{0\}\right) = j_{t-1}\]
by the very definition of $(j_i)_{i=1}^m$, so it follows that
\begin{align*}
h_{\beta^{+c_t}}(c_t)&=|\{(i',j')\in D(\beta^{+c_t}):i'\le\beta_{j_t},\;j_{t-1}<j'\le j_t\}| \\
&=|\{(i',j')\in D(\beta):j_{t-1}<j'\le j_t\}|.
\end{align*}
The intervals $\{(j_{t-1},j_t]\}_{t=1}^m$ partition $[\ell(\beta)]$. Thus, summing $h_{\beta^{+c_t}}(c_t)$ over $t\in[m]$ yields $n-1$, and adding $h_{\beta^{+c_{m+1}}}(c_{m+1})=1$ gives the final sum of $n$.

For the second identity, observe that the condition $w(\beta^{+c_t})=w(\beta)+1$ is equivalent to $\beta_{j_t}=w(\beta)$, so if we let $j_p:=j_{\max}(\beta)$, then the expression simplifies to
\[\sum_{\substack{c\,\in\,\out(D(\beta)) \\ w(\beta^{+c})=w(\beta)+1}} h_{\beta^{+c}}(c) = \sum_{\substack{c\,\in\,\out(D(\beta)) \\ t\in [p]}} h_{\beta^{+c_{t}}}(c_t).\]
The intervals $\{(j_{t-1},j_t]\}_{t=1}^p$ partition $[j_{\max}(\beta)]$, and again,  summing $h_{\beta^{+c_t}}(c_t)$ over $t\in[p]$ yields $t(\beta)$.
\end{proof}

With \Cref{lem:corner} and \Cref{lem:hooksum}, we express the difference $\mathcal{D}_n' - \mathcal{D}_{n-1}'$ cleanly in terms of $N(\beta)$ and $t(\beta)$ for $\beta\models n-1$.
\begin{theorem}\label{thm:increment}
For all $n\ge2$,
\[\mathcal{D}'_n = \mathcal{D}'_{n-1} + \frac{1}{n!}\sum_{\beta\models n-1} N(\beta)\,t(\beta).\]
\end{theorem}

\begin{proof}
First, by \Cref{lem:corner}, we deduce
\begin{align*}
\mathcal{D}'_n
  &= \frac{1}{n!}\sum_{\alpha\models n}N(\alpha)\,w(\alpha) \\
  &= \frac{1}{n!}\sum_{\substack{\alpha\models n \\ c\,\in\,\inn(D(\alpha))}} N(\alpha^{-c})\,h_\alpha(c)\,w(\alpha) \\
  &= \frac{1}{n!}\sum_{\substack{\beta\models n-1 \\ c\in \out(D(\beta))}} N(\beta)\,h_{\beta^{+c}}(c)\,w(\beta^{+c}).
\end{align*}
Now observe that $w(\beta^{+c})$ is either $w(\beta)$ or $w(\beta)+1$, so by casework and invoking \Cref{lem:hooksum}, the expression further reduces to
\begin{align*}
&\frac{1}{n!}\sum_{\beta\models n-1}N(\beta)\left(w(\beta)\sum_{c\in \out(D(\beta))}h_{\beta^{+c}}(c) + \sum_{\substack{c\,\in\,\out(D(\beta)) \\ w(\beta^{+c})=w(\beta)+1}} h_{\beta^{+c}}(c)\right) \\
= &\frac{1}{n!}\sum_{\beta\models n-1}N(\beta)\bigl(w(\beta)\, n+t(\beta)\bigr) \\
= &\mathcal{D}'_{n-1} + \frac{1}{n!}\sum_{\beta\models n-1}N(\beta)\,t(\beta).
\end{align*}
\end{proof}

\begin{lemma}\label{lem:floor}
The sequence $(\mathcal{D}'_{n-1}/n)_{n=2}^{\infty}$ is monotonically increasing.
\end{lemma}
\begin{proof}
Since all parts of $\beta$ are positive, $t(\beta)=\sum_{j=1}^{j_{\max}(\beta)}\beta_j\ge\beta_{j_{\max}(\beta)}=w(\beta)$ for all $\beta\models n-1$. Substituting this inequality into \Cref{thm:increment} yields
\[\mathcal{D}'_n \;\ge\; \mathcal{D}'_{n-1} + \frac{1}{n!}\sum_{\beta\models n-1}N(\beta)\,w(\beta)=\; \frac{n+1}{n}\,\mathcal{D}'_{n-1}\]
for all $n\ge2$.
\end{proof}

\begin{proof}[Proof of \Cref{thm:main}]
By \Cref{lem:floor}, $(\mathcal{D}'_{n-1}/n)_{n=2}^{\infty}$ is a monotonically increasing sequence bounded above by $1$, so the limit \[L^*:=\lim_{n\to\infty}\frac{\mathcal{D}_{n-1}'}{n}\] exists.  It shortly follows that
\[\lim_{n\to\infty}\frac{\mathcal{D}'_n}{n} = \lim_{n\to\infty}\frac{n+1}{n}\cdot\frac{\mathcal{D}_{n}'}{n+1} = L^*\]
and
\[\lim_{n\to\infty}\frac{\mathcal{D}_n}{n} = \lim_{n\to\infty}\frac{\mathcal{D}'_n}{n} + \lim_{n\to\infty}\left(\frac{\mathcal{D}_n}{n}-\frac{\mathcal{D}'_n}{n}\right) = L^*.\]
\end{proof}

\Cref{cor:lb} follows immediately from \Cref{lem:floor} and the proof of \Cref{thm:main}. Below, we provide the computed values of $\mathcal{D}'_{n-1}/n$ for $n\in [2, 38]$. Every element of the sequence is a lower bound for $L^*$.

\begin{table}[ht]
\centering
\renewcommand{\arraystretch}{1.2}
\begin{tabular}{r c | r c | r c | r c}
\toprule
$n$ & $\mathcal{D}'_{n-1}/n$ & $n$ & $\mathcal{D}'_{n-1}/n$ & $n$ & $\mathcal{D}'_{n-1}/n$ & $n$ & $\mathcal{D}'_{n-1}/n$ \\
\midrule
 2 & $0.500000$ & 12 & $0.649764$ & 22 & $0.693840$ & 32 & $0.715554$ \\
 3 & $0.500000$ & 13 & $0.656212$ & 23 & $0.696631$ & 33 & $0.717158$ \\
 4 & $0.541667$ & 14 & $0.662037$ & 24 & $0.699250$ & 34 & $0.718689$ \\
 5 & $0.566667$ & 15 & $0.667288$ & 25 & $0.701709$ & 35 & $0.720153$ \\
 6 & $0.586111$ & 16 & $0.672084$ & 26 & $0.704027$ & 36 & $0.721554$ \\
 7 & $0.601190$ & 17 & $0.676460$ & 27 & $0.706214$ & 37 & $0.722895$ \\
 8 & $0.614137$ & 18 & $0.680492$ & 28 & $0.708283$ & 38 & $0.724182$ \\
 9 & $0.624909$ & 19 & $0.684206$ & 29 & $0.710242$ &    &            \\
10 & $0.634330$ & 20 & $0.687653$ & 30 & $0.712103$ &    &            \\
11 & $0.642494$ & 21 & $0.690853$ & 31 & $0.713870$ &    &            \\
\bottomrule
\end{tabular}
\end{table}

\section{Future directions}\label{sec:future}

We conclude by discussing several potential future directions. First, numerical evidence suggests that the sequence $(\mathcal{D}_{n}' - \mathcal{D}'_{n-1})_{n=2}^{\infty}$ is monotonically increasing, so we accordingly make \Cref{conj:r}. 
\begin{conjecture}\label{conj:r}
The sequence $(\mathcal{D}_{n}' - \mathcal{D}'_{n-1})_{n=2}^{\infty}$ is monotonically increasing.
\end{conjecture}
Recall that by \Cref{thm:increment} the difference $\mathcal{D}_{n}'-\mathcal{D}'_{n-1}$ is proportional to the expected value of $t(\alpha_{\pi})$ for $\pi\in S_{n-1}'$. We believe that \Cref{conj:r} is substantially more difficult than \Cref{thm:main} because it requires more information about the distribution of $D(\alpha_{\pi})$ beyond the crude bound $t(\alpha_{\pi})\ge w(\alpha_{\pi})$.

More generally, it would be nice to obtain a bound for 
\[\left|L^*-\frac{\mathcal{D}'_{n-1}}{n}\right|\]
in terms of $n$ to understand the convergence rate of $(\mathcal{D}'_{n-1}/n)_{n=2}^{\infty}$. It might also be fruitful to explore sharper asymptotics for $\mathcal{D}_n$.

\bibliographystyle{plain}
\bibliography{bib_monotonicity}

\end{document}